\documentclass[12pt,a4paper]{amsart}

\usepackage[margin=25mm,footskip=1.0cm]{geometry}
\usepackage{amsmath,amstext,amsopn,amssymb,amsthm,amsaddr,color,lipsum}
\usepackage[utf8]{inputenc}

\usepackage{url}

\usepackage{tabu}
\usepackage{ccaption}
\usepackage{etoolbox}
\usepackage{stringstrings}
\usepackage{listofitems}
\usepackage{mathtools}
\usepackage{enumitem}

\usepackage{scrextend}
\deffootnote{1.5em}{0em}{${}^{\text{\thefootnotemark}}$\hspace*{0.5em}}

\makeatletter
\renewcommand{\tocsection}[3]{\indentlabel{\@ifnotempty{#2}{\bfseries\ignorespaces#1 #2\rlap{.}\quad}}\bfseries\parbox[t]{0.7\linewidth}{#3}\vspace{0.5\baselineskip}}
\renewcommand{\tocsubsection}[3]{\setsepchar{.}\readlist\mysectionnumber{#2}\hphantom{\bfseries\ignorespaces\mysectionnumber[1]\quad}\indentlabel{\@ifnotempty{#2}{\ignorespaces#1 #2\quad}}#3\vspace{0.5\baselineskip}}
\newcommand\@dotsep{2.5}
\def\@tocline#1#2#3#4#5#6#7{\relax
\ifnum #1>\c@tocdepth 
\else
\par \addpenalty\@secpenalty\addvspace{#2}%
\begingroup \hyphenpenalty\@M
\@ifempty{#4}{\@tempdima\csname r@tocindent\number#1\endcsname\relax}{\@tempdima#4\relax}%
\parindent\z@ \leftskip#3\relax \advance\leftskip\@tempdima\relax
\rightskip\@pnumwidth plus1em \parfillskip-\@pnumwidth
#5\leavevmode\hskip-\@tempdima{#6}\nobreak
\ifnum#1=1\else\leaders\hbox{$\m@th\mkern \@dotsep mu\hbox{.}\mkern \@dotsep mu$}\fi
\hfill
\nobreak
\hbox to\@pnumwidth{\@tocpagenum{\ifnum#1=1\bfseries\fi#7}}\par
\nobreak
\endgroup
\fi}
\AtBeginDocument{\expandafter\renewcommand\csname r@tocindent0\endcsname{0pt}}
\makeatother

\makeatletter
\renewcommand{\@secnumfont}{\bfseries}
\renewcommand\section{\@startsection{section}{1}{0pt}{-1\baselineskip}{0.2\baselineskip}{\raggedright\normalfont\large\bfseries}}
\renewcommand\subsection{\@startsection{subsection}{2}{0pt}{-1\baselineskip}{0.2\baselineskip}{\raggedright\normalfont\bfseries}}
\renewcommand\subsubsection{\@startsection{subsubsection}{3}{0pt}{-1\baselineskip}{0.1\baselineskip}{\raggedright\normalfont\bfseries}}

\makeatother

\makeatletter
\AtBeginDocument{%
\abovedisplayskip=2ex
\belowdisplayskip=2ex
\abovedisplayshortskip=2ex
\belowdisplayshortskip=2ex
}
\makeatother

\numberwithin{equation}{section}

\newtheoremstyle{theorem}{4ex}{4ex}{\it}{}{\bf}{}{1em}{}
\theoremstyle{theorem}
\newtheorem{theorem}{Theorem}
\newtheorem{lemma}{Lemma}
\newtheoremstyle{definition}{2ex}{2ex}{\it}{}{\bf}{}{1em}{}
\theoremstyle{definition}
\newtheorem{defi}{Definition}
\newtheoremstyle{remark}{2ex}{2ex}{\normalfont}{}{\bf}{}{1em}{}
\theoremstyle{remark}
\newtheorem{remark}{Remark}
\newtheoremstyle{proposition}{4ex}{4ex}{\it}{}{\bf}{}{1em}{}
\theoremstyle{proposition}
\newtheorem{propo}{Proposition}
\numberwithin{propo}{section}

\newcommand{\ol}{\overline}

\newcommand{\ve}{\varepsilon}

\newcommand{\N}{\mathbb{N}}

\newcommand{\R}{\mathbb{R}}

\newcommand{\mcE}{\mathcal{E}}
\newcommand{\mcF}{\mathcal{F}}
\newcommand{\mcN}{\mathcal{N}}

\newcommand{\dd}{\,\mathrm{d}}
\newcommand{\del}{\partial}
\newcommand{\intl}{\int\limits}

\newcommand{\scal}[1]{\left\langle #1 \right\rangle}
\newcommand{\nor}[1]{\left\lVert #1 \right\rVert}

\DeclareMathOperator{\curl}{curl}
\DeclareMathOperator{\Div}{div}
\DeclareMathOperator*{\esssup}{ess\,sup}
\DeclareMathOperator{\mes}{mes}
\DeclareMathOperator{\supp}{supp}

\begin{document}

\title{On some properties of weak solutions to the                                                                                                                                     Maxwell equations}
\author{Joachim Naumann}
\address{Institut f\"ur Mathematik\\Humboldt-Universit\"at zu Berlin\\[2ex]{\normalfont\today}}
\email{jnaumann@math.hu-berlin.de}

\begin{abstract}
This paper is concerned with weak solutions $\{e,h\} \in L^2 \times L^2$ of the time-dependent Maxwell equations.
We show that these solutions obey an energy equality. 
Our method of proof is based on the approximation of $\{e,h\}$  by its Steklov mean with respect to time t. This approximation technique  is well-known for establishing integral estimates for weak solutions of parabolic equations. In addition we prove the 
uniqueness of $\{e,h\}$. 
\end{abstract}

\keywords{Maxwell equations, electromagnetic energy, weak solutions, energy equality, Steklov mean. 2020 MSC: 35A01, 35A02, 35B45,
35Q61.}

\maketitle

\tableofcontents

\newpage

\section{Introduction}\label{s:1}

\subsection{Field equations.~Constitutive laws}\label{s:1.1}

Let $\Omega \in \R^3$ be a bounded domain, let $0 < T < +\infty$ and put $Q_T = \Omega \times ]0,T[$.
The evolution of an electromagnetic field in a medium at rest occupying the region $\Omega$, is governed by the system of PDEs
\begin{align}
\label{eq:1.1}
\del_t d + j\,\, & =\,\, \curl h && \text{Amp\`ere-Maxwell law,}\\
\label{eq:1.2}
\del_t b\,\, & =\,\, -\curl e && \text{Faraday law.}
\end{align}
The meaning of the vector fields $\{d,b;e,h\}$ is\\[1ex]
\hspace*{6em}$d$ \textit{electric displacement},
$b$ \textit{magnetic induction},\\
\hspace*{6em}$e$ \textit{electric field},
$h$ \textit{magnetic field}.\\[1ex]
The scalar fields\\[1ex]
\hspace*{6em}$d \cdot e$ \textit{electric density},\\
\hspace*{6em}$b \cdot h$ \textit{magnetic density},\\
\hspace*{6em}$j \cdot e$ \textit{electric power density}\\[1ex]
are the basic energy densities for the electromagnetic field modelled by $\{d,b;e,h\}$.

Applying the $\Div$-operator to both sides of (\ref{eq:1.1}) and (\ref{eq:1.2}) and integrating over the interval $[0,t]$ ($0 < t \leq T$) gives
\begin{align*}
(\Div d)(x,t) + \int_0^t (\Div j)(x,s)\dd s\,\, =\,\, (\Div d)(x,0),\\
(\Div b)(x,t)\,\, =\,\, (\Div b)(x,0)
\end{align*}
for all $(x,t) \in Q_T$. The scalar function
\begin{align*}
\rho(x,t)\,\, =\,\, -\int_0^t (\Div j)(x,s) \dd s
\end{align*}
is called \textit{electric charge}.
For details see, e.g., \cite[Chap.~1]{Bossavit}, \cite[Chap.~1]{Fabrizio}, \cite[Chap.~18; 27]{Feynman}, \cite[Chap.~6]{Jackson} and \cite[Teil~I, §§3;4]{Sommerfeld}.

In this paper, we consider the following constitutive laws for the vector fields $\{d,b;e,h\}$
\begin{align}
\label{eq:1.3}
d\,\, =\,\, \ve e, \quad b\,\, =\,\, \mu h~\text{\footnotemark},
\end{align}
\footnotetext{For $a = \{a_{kl}\}_{k,l=1,2,3}$ and $\xi = \{\xi_k\}_{k=1,2,3}$ ($a_{kl}, \xi_k \in \R$) we write $a\xi = \{a_{kl}\xi_l\}_{k=1,2,3}$ (summation over repeated indices).}%
where the symmetric non-negative matrices $\ve = \ve(x) = \{\ve_{kl}(x)\}_{k,l=1,2,3}$ and $\mu = \mu(x) = \{\mu_{kl}(x)\}_{k,l=1,2,3}$ ($x \in \Omega$)
characterize the electric permittivity and magnetic permeability, respectively, of the medium under consideration.
More general constitutive laws are discussed, e.g., in \cite[Section~1.4]{Fabrizio}, and \cite[S.~20-22]{Sommerfeld}.

Substituting (\ref{eq:1.3}) into (\ref{eq:1.1}), (\ref{eq:1.2}) gives
\begin{align}
\label{eq:1.4}
\ve\del_te + j\,\, & =\,\, \curl h,\\
\label{eq:1.5}
\mu\del_th\,\, & =\,\, - \curl e.
\end{align}
Here, the vector fields $e$ and $h$ are the unknowns.

\begin{remark}
\label{r:1}
The following structure of $j$ is widely considered
\begin{align*}
j\,\, =\,\, j_0 + j_1,
\end{align*}
where
\begin{align*}
j_0\,\, =\,\, \sigma e \quad \text{Ohm law}.
\end{align*}
The matrix $\sigma = \sigma(x,t) = \{\sigma_{kl}(x,t)\}_{k,l=1,2,3}$ ($(x,t) \in Q_T$) characterizes the \textit{electrical conductivity}
of the medium.
By physical reasons, $\sigma$ has to satisfy the conditions
\begin{align*}
(\sigma(x,t)\xi) \cdot \xi\, \geq\, 0\,\,\, \ \forall (x,t) \in Q_T,  \ \forall \xi \in \R^3
\end{align*}
The vector field $j_1$ represents a given current density (see, e.g., \cite[pp.~10-11]{Bossavit}, \cite[S.~19-20]{Sommerfeld}).
\end{remark}

\subsection{Balance of electromagnetic energy}\label{s:1.2}

Let $\{e,h\}$ be a classical solution of (\ref{eq:1.4}), (\ref{eq:1.5}) in $Q_T$.
We multiply scalarly (\ref{eq:1.4}) and (\ref{eq:1.5}) by $e$ and $h$, respectively, and add the equations.
We get
\begin{align}
\label{eq:1.6}
\frac{1}{2} \frac{\del}{\del t} \bigg((\ve e) \cdot e + (\mu h) \cdot h\bigg) + \Div S + j \cdot h\,\, =\,\, 0~\text{\footnotemark},
\end{align}
\footnotetext{\label{fn:2}Notice $\Div(a \times b) = b \cdot \curl a - a \cdot \curl b$.}%
where
\begin{align*}
S\,\, =\,\, e \times h
\end{align*}
denotes the \textit{Poynting vector} of $\{e,h\}$.
Integrating (\ref{eq:1.6}) over $\Omega$ yields
\begin{align}
\label{eq:1.7}
\frac{\dd}{\dd t}\mcE(t) + \intl_{\Omega} (\Div S)(x,t) \dd x + \intl_{\Omega} j(x,t) \cdot e(x,t) \dd x\,\, =\,\, 0
\end{align}
for all $t \in [0,T]$, where
\begin{align*}
\mcE(t)\,\, =\,\, \frac{1}{2} \intl_{\Omega} \big[(\ve(x)e(x,t)) \cdot e(x,t) + (\mu(x)h(x,t)) \cdot h(x,t)\big] \dd x, \quad t \in [0,T].
\end{align*}
The scalar $\mcE(t)$ represents the \textit{electromagnetic energy of $\{e,h\}$} at the time $t$.
Equation (\ref{eq:1.7}) is called \textit{balance of electromagnetic energy} (or \textit{Poynting's theorem}).

\qed

Throughout our further discussion in this section, we suppose that the boundary $\Gamma = \del\Omega$ of $\Omega$ is sufficiently smooth.
An application of the divergence theorem to the integral involving $\Div S$ in (\ref{eq:1.7}) gives
\begin{align}
\label{eq:1.8}
\frac{\dd}{\dd t} \mcE(t) + \intl_{\Gamma} n(x) \cdot S(x,t) \dd \Gamma + \intl_{\Omega} j(x,t) \cdot e(x,t) \dd x\,\, =\,\, 0,
\end{align}
where $n(x)$ denotes the outward directed unit normal at $x \in \Gamma$.
The boundary integral in (\ref{eq:1.8}) characterizes the outgoing flux of electric power through $\Gamma$
(see, e.g., \cite[Chap.~1.3.1]{Bossavit}, \cite[Chap.~6.8]{Jackson}, and \cite[Teil~I, §5]{Sommerfeld}).

We note that (\ref{eq:1.8}) is formally equivalent to the following \textit{energy equality}
\begin{align}
\label{eq:1.9}
\mcE(t) + \intl_s^t \intl_{\Gamma} n(x) \cdot S(x,\tau) \dd \Gamma \dd \tau + \intl_s^t \intl_{\Omega} j(x,\tau) \cdot e(x,\tau) \dd x \dd \tau\,\, =\,\, \mcE(s)
\end{align}
for all $s,t \in [0,T]$, $s < t$.

\qed

We now consider a classical solution $\{e,h\} \in C^1(\ol{Q}_T)^3 \times C^1(\ol{Q}_T)^3$ of (\ref{eq:1.4}), (\ref{eq:1.5}) that satisfies the conditions
\begin{align}
\label{eq:1.10}
& n \times e\,\, =\,\, 0\quad \text{ on } \Gamma \times [0,T],\\
\label{eq:1.11}
& e\,\, =\,\, e_0,\quad h\,\, =\,\, h_0\quad \text { in } \Omega \times \{0\},
\end{align}
where $\{e_0,h_0\}$ are given data in $\Omega$.
We obtain
\begin{align*}
\intl_{\Gamma} n \cdot S \dd \Gamma
\,\,& =\,\, \intl_{\Gamma} n \cdot (e \times h) \dd \Gamma \quad \text{(definition of $S$)}\\[1ex]
& =\,\, \frac{1}{2} \intl_{\Gamma} \left( - (n \times h) \cdot e + (n\times e) \cdot h \right) \dd \Gamma~\text{\footnotemark}\\[1ex]
& =\,\, 0.
\end{align*}
\footnotetext{Observe $e \cdot (h \times n) = n \cdot (e \times h) = h \cdot (n \times e)$ (cf. footnote \ref{fn:2})}%
Thus, energy equality (\ref{eq:1.9}) takes the form
\begin{align}
\label{eq:1.12}
\mcE(t) + \intl_s^t \intl_{\Omega} j(x,\tau) \cdot e(x,\tau) \dd x \dd \tau\,\, =\,\, \mcE(s) \quad \forall s, t \in [0,T],\ s \leq t,
\end{align}
or, equivalently,
\begin{align}
\label{eq:1.13}
\mcE(t) + \intl_0^t \intl_{\Omega} j(x,\tau) \cdot e(x,\tau) \dd x \dd \tau\,\, =\,\, \mcE(0) \quad \forall t \in [0,T],
\end{align}
where
\begin{align*}
\mcE(0)\,\, =\,\, \frac{1}{2} \intl_{\Omega} \left[ (\ve(x)e_0(x)) \cdot e_0(x) + (\mu(x)h_0(x)) \cdot h_0(x) \right] \dd x.
\end{align*}

\qed

\subsection{An integral identity for classical solutions of (\ref{eq:1.4}), (\ref{eq:1.5}), (\ref{eq:1.10}), (\ref{eq:1.11})}\label{s:1.3}

Let $\{e,h\} \in C^1(\ol{Q}_T)^3 \times C^1(\ol{Q}_T)^3$ be a classical solution of (\ref{eq:1.4}), (\ref{eq:1.5}).
Given $\{\phi,\psi\} \in C^1(\ol{Q}_T)^3 \times C^1(\ol{Q}_T)^3$ such that $\phi(x,T) = \psi(x,T) = 0$ for all $x \in \Omega$,
we multiply (\ref{eq:1.4}) and (\ref{eq:1.5}) scalarly by $\phi$ and $\psi$, respectively, add the obtained equations,
integrate over $Q_T$ and carry out an integration by parts with respect to $t$ over the interval $[0,T]$.
It follows
\begin{align}
\label{eq:1.14}
& - \intl_{Q_T} \left( (\ve e) \cdot \del_t\phi + (\mu h) \cdot \del_t\psi \right) \dd x \dd t
+ \intl_{Q_T} \left( - (\curl h) \cdot \phi + (\curl e) \cdot \psi \right) \dd x \dd t \notag\\
& \hspace*{1em} +\, \intl_{Q_T} j \cdot \phi \dd x \dd t\,\, = \intl_{\Omega} \left( (\ve e) \cdot \phi + (\mu h) \cdot \psi \right) \dd x \bigg|_{t=0}.
\end{align}

To proceed further, we will need the following Green formula
\begin{align*}
\intl_{\Omega} (\curl a) \cdot b \dd x - \intl_{\Omega} a \cdot (\curl b) \dd x\,\, =\, \intl_{\Gamma} (n \times a) \cdot b \dd \Gamma,\quad a, b \in C^1(\ol{\Omega})^3.
\end{align*}
We make use of this formula with $a = -\,h(\cdot,t)$, $b = \phi(\cdot,t)$ resp. $a = e(\cdot,t)$, $b = \psi(\cdot,t)$ ($t \in [0,T]$), and integrate then over the interval $[0,T]$.
We obtain
\begin{align*}
& \hspace*{0.2\linewidth} \intl_{Q_T} \left( - (\curl h) \cdot \phi + (\curl e) \cdot \psi \right) \dd x \dd t\\
& =\, \intl_0^T \intl_{\Gamma} \left( - (n \times h) \cdot \phi + (n \times e) \cdot \psi \right) \dd \Gamma \dd t
+ \intl_{Q_T} \left( - h \cdot (\curl \phi) + e \cdot (\curl \psi) \right) \dd x \dd t.
\end{align*}
Substituting this into (\ref{eq:1.14}) gives
\begin{align*}
&-\intl_{Q_T}\left( (\ve e)\cdot\del_t \phi + (\mu h)\cdot\del_t\psi\right) \dd x \dd t
+ \intl_0^T \intl_{\Gamma} \left( - (n \times h) \cdot \phi + (n \times e) \cdot \psi \right) \dd \Gamma \dd t \notag\\
& \hspace*{1em} + \intl_{Q_T} \left( - h \cdot (\curl \phi) + e \cdot (\curl \psi) \right) \dd x \dd t
+ \intl_{Q_T} j \cdot \phi \dd x \dd t \notag\\
& =\, \intl_{\Omega} \left( (\ve e) \cdot \phi + (\mu h) \cdot \psi \right) \dd x \bigg|_{t=0}.
\end{align*}

\qed

Thus, since $-(n \times h) \cdot \phi = (n \times \phi) \cdot h$ on $\Gamma$,
it follows that every classical solution $\{e,h\}$ of (\ref{eq:1.4}), (\ref{eq:1.5}), (\ref{eq:1.10}), (\ref{eq:1.11}) satisfies the \textit{integral identity}
\begin{align}
\label{eq:1.15}
& - \intl_{Q_T} \left( (\ve e) \cdot \del_t\phi + (\mu h) \cdot \del_t\psi \right) \dd x \dd t \notag\\
& \hspace*{1em} + \intl_{Q_T} \left( - h \cdot (\curl \phi) + e \cdot (\curl \psi) \right) \dd x \dd t
+ \intl_{Q_T} j \cdot \phi \dd x \dd t \notag\\
& =\, \intl_{\Omega} \left[ (\ve(x)e_0(x)) \cdot \phi(x,0) + (\mu(x)h_0(x)) \cdot \psi(x,0) \right] \dd x
\end{align}
for all $\{\phi,\psi\} \in C^1(\ol{Q_T})^3 \times C^1(\ol{Q_T})^3$ such that

\pagebreak

\begin{itemize}
\setlength{\itemsep}{2ex}
\item\, $\phi\,\,=\,\, \psi =\,\, 0\quad\text{ in }\Omega \times \{T\}$,
\item\, $n \times \phi\,\, =\,\, 0\quad\text{ on }\Gamma \times [0,T]$.
\end{itemize}
\qed

Integral identity (\ref{eq:1.15}) evidently continues to make sense for $\{e,h\} \in L^2(Q_T)^3 \times L^2(Q_T)^3$.
This motivates the definition of the notion of weak solution of (\ref{eq:1.4}), (\ref{eq:1.5}), (\ref{eq:1.10}), (\ref{eq:1.11}) we will give in the following section.

\section{Definition and basic properties of weak solutions of (\ref{eq:1.4}), (\ref{eq:1.5}), (\ref{eq:1.10}), (\ref{eq:1.11})}\label{s:2}

Let $\Omega \subset \R^3$ be a bounded domain.
We introduce the space
\begin{align*}
V :=\, \left\{ u \in L^2(\Omega)^3; \ \curl u \in L^2(\Omega)^3 \right\}.
\end{align*}
$V$ is a Hilbert space with respect to the scalar product
\begin{align*}
(u,v)_V := \intl_{\Omega} (u \cdot v + (\curl u) \cdot (\curl v)) \dd x
\end{align*}
We next define
\begin{align*}
V_0 :=\, \left\{ u \in V; \ \intl_{\Omega} (\curl u) \cdot z \dd x\, = \intl_{\Omega} u \cdot (\curl z) \dd x \quad \forall z \in V \right\}.
\end{align*}
To our knowledge, this space has been introduced for the first time in \cite{Leis}, and then used  in other papers, e.g., \cite{Jochmann}, \cite{Weber}.
The vector fields $u \in V_0$ satisfy the condition $n \times u = 0$ on the boundary $\Gamma = \del\Omega$ of $\Omega$
in a generalized sense.
More specifically, if $\Gamma$ is Lipschitz continuous, then there exists a linear continuous mapping $\gamma_\tau: V \to H^{-1/2}(\Gamma)^3$ such that
\begin{align*}
\gamma_\tau(u)\,\, =\,\, n \times u |_{\Gamma} \quad \forall u \in C^1(\ol{\Omega})^3,
\end{align*}
\begin{align*}
\intl_{\Omega} (\curl u) \cdot \varphi \dd x - \intl_{\Omega} u \cdot (\curl \varphi) \dd x
\,\,=\,\, \scal{\gamma_\tau(u),\varphi}_{(H^{1/2})^3} \quad \forall u \in V, \ \forall \varphi \in H^1(\Omega)^3~\text{\footnotemark}
\end{align*}
\footnotetext{$\scal{\cdot,\cdot}_{(H^{1/2})^3} = $ dual pairing between $H^{-1/2}(\Gamma)^3$ and $H^{1/2}(\Gamma)^3$ (cf. also below).}%
(cf., e.g., \cite{Amrouche}, \cite[Chap.~IX, 2.]{Dautray}, \cite[Chap.~I, Th.~2.11]{Girault}).
It follows
\begin{align*}
V_0 := \left\{ u \in V; \gamma_\tau(u)\, =\, 0 \text{ in } H^{-1/2}(\Gamma)^3 \right\}.
\end{align*}
Based upon this result, in Appendix II we give an equivalent characterization of $V_0$.

\qed

For our discussion of weak solutions of (\ref{eq:1.4}), (\ref{eq:1.5}), (\ref{eq:1.10}), (\ref{eq:1.11}) we will need spaces of functions from the interval $[0,T]$ into a real normed vector space $X$.

Let $|\cdot|_X$ denote the norm in $X$.
By $L^p(0,T;X)$ ($1 \leq p < +\infty$) we denote the vector space of all equivalence classes of strongly (Bochner) measurable functions $u: [0,T] \to X$ such that $t \mapsto |u(t)|_X$ is in $L^p(0,T)$.
A norm on $L^p(0,T;X)$ is given by
\begin{align*}
\nor{u}_{L^p(0,T;X)} :=
\begin{cases}
\,\,\displaystyle \left( \int_0^T |u(t)|_X^p \right)^{1/p} & \text{if }\, 1 \leq\, p\, < +\infty,\\[2ex]
\displaystyle \esssup\limits_{t \in ]0,T[} |u(t)|_X & \text{if }\, p\, =\, +\infty.
\end{cases}
\end{align*}
For more details see, e.g., \cite[Chap.~4]{Bourbaki}, \cite[Chap.~1]{Droniou} and \cite[Chap.~23.2, 23.3]{Zeidler}.

If $H$ denotes a real Hilbert space with scalar product $(\cdot,\cdot)_H$, then $L^2(0,T;H)$ is a Hilbert space for the scalar product
\begin{align*}
(u,v)_{L^2(0,T;H)} := \int_0^T (u(t),v(t))_H \dd t.
\end{align*}

Given $u \in L^p(Q_T)$ $(1 \leq p < +\infty)$, we define
\begin{align*}
[u](t) :=\,\ u(\cdot,t)\,\,\,\, \text{ for a.e. } t \in [0,T].
\end{align*}
By Fubini's theorem, $[u](\cdot) \in L^p(0,T)$ and
\begin{align*}
\intl_{Q_T} |u(x,t)|^p \dd x \dd t \,\,= \int_0^T \nor{[u](t)}_{L^p(\Omega)}^p \dd t.
\end{align*}
It is easy to prove that the mapping $u \mapsto [u]$ is a linear isometry from $L^p(Q_T)$ onto $L^p(0,T;L^p(\Omega))$.
Throughout our paper we identify these spaces.

\qed

For what follows, we suppose that
\begin{align}
\label{eq:2.1}
& \begin{cases}
&\text{the entries of the matrices } \ve \,=\, \ve(x), \ \mu\, =\, \mu(x)\\
&\text{are bounded measurable functions in } \Omega;
\end{cases}\\[1\baselineskip]
\label{eq:2.2}
& j \in L^2(Q_T)^3, \ e_0, h_0 \in L^2(\Omega)^3.
\end{align}
The following definition extends the integral identity (\ref{eq:1.15}) to the $L^2$-framework.
\begin{defi}
Assume (\ref{eq:2.1}), (\ref{eq:2.2}).
The pair
\begin{align*}
\{e,h\} \in L^2(Q_T)^3 \times L^2(Q_T)^3
\end{align*}
is called {\em weak solution of} (\ref{eq:1.4}), (\ref{eq:1.5}), (\ref{eq:1.10}), (\ref{eq:1.11}) if
\pagebreak
\begin{align}
\label{eq:2.3}
\begin{cases}
&\displaystyle
- \intl_{Q_T} \left( (\ve e) \cdot \del_t\phi + (\mu h) \cdot \del_t\psi \right) \dd x \dd t\\
&\displaystyle
+ \intl_{Q_T} \left( - h \cdot (\curl \phi) + e \cdot (\curl \psi) \right) \dd x \dd t
+ \intl_{Q_T} j \cdot \phi \dd x \dd t\\[5ex]
&\displaystyle
= \intl_{\Omega} \left[ (\ve(x)e_0(x)) \cdot \phi(x,0) + (\mu(x)h_0(x)) \cdot \psi(x,0) \right] \dd x\\[5ex]
&\text{for all } \{\phi,\psi\} \in L^2(0,T;V_0) \times L^2(0,T;V) \text{ such that }\\
&\del_t\phi, \del_t\psi \in L^2(Q_T)^3, \ \phi(\cdot,T) = \psi(\cdot,T) = 0 \text{ a.e. in } \Omega.
\end{cases}
\end{align}
\end{defi}
From our discussion in Section \ref{s:1} it follows that every classical solution of (\ref{eq:1.4}), (\ref{eq:1.5}), (\ref{eq:1.10}), (\ref{eq:1.11}) is a weak solution of this problem. We notice that (\ref{eq:2.3}) basically coincides with the definition of weak solutions of initial-boundary value problems for the Maxwell equations
that is introduced in \cite[Chap.~VII,4.2]{Duvaut}, \cite[p.~326]{Fabrizio} and \cite{Jochmann}.

Existence theorems for weak solutions of (1.4), (1.5), (1.10), (1.11) are presented in [7, Chap. VII,4.3)], [8], [9, Section 7.8.3] and [13].

\qed

\textit{Existence of the distributional derivatives $(\ve e)'$,$(\mu h)'$.}
We introduce more notations.

Let $X$ be a real normed vector space with norm $|\cdot|_X$.
By $X^*$ we denote the dual space of $X$, and by $\scal{x^*,x}_X$ the dual pairing between $x^* \in X^*$ and $x \in X$. Let $H$ be a real Hilbert space with scalar product $(\cdot,\cdot)_H$ such that $X \subset H$ continuously and densely.
Identifying $H$ with its dual space $H^*$ via the Riesz representation theorem, we obtain
\begin{align}
\label{eq:2.4}
H \subset X^* \text{ continuously}, \quad \scal{z,x}_X\, =\, (z,x)_H \quad \forall z \in H, \ \forall x \in X
\end{align}
(cf., e.g., \cite[Chap.~23.4]{Zeidler}).
If $X$ is reflexive, then $H \subset X^*$ densely.

Next, let $X$ and $Y$ be two real normed vector spaces such that $X \subset Y$ continuously and densely.
Given $u \in L^1(0,T;X)$, we identify $u$ with an element in $L^1(0,T;Y)$ and denote this element again by $u$.
If there exists $U \in L^1(0,T;Y)$ such that
\begin{align*}
\int_0^T \dot{\zeta}(t) u(t) \dd t \,\stackrel{\text{in } Y}{=} - \int_0^T \zeta(t) U(t) \dd t \quad \forall \zeta \in C_c^\infty(]0,T[),
\end{align*}
then $U$ will be called the {\em derivative of $u$ in the sense of distributions from $]0,T[$ into $Y$} and denoted by
\begin{align*}
u' :=\, U
\end{align*}
(see, e.g., \cite[Appendice]{Brezis}, \cite[Chap.~21]{Droniou}, \cite[Chap.~1.3]{Lions} and \cite[Chap.~23.5, 23.6]{Zeidler}).
$u'$ is uniquely determined by $u$.

The existence of the distributional derivative $u' \in L^1(0,T;Y)$ is equivalent to the existence of a function $\tilde{u}: [0,T] \to Y$ such that
\begin{itemize}
\setlength{\itemsep}{2ex}
\item\, $\displaystyle \tilde{u}(t)\, \stackrel{\text{in } Y}{=} u(t)\,\,\,\,$ for a.e. $t \in [0,T]$,
\item\, $\displaystyle \tilde{u}(t)\, =\, \tilde{u}(0) + \int_0^t u(s) \dd s \quad \forall t \in [0,T]$,
\end{itemize}
i.e., $\tilde{u}$ is the absolutely continuous representative of equivalence class $u$ (cf.~\cite[Appendice]{Brezis}).

Let $X$ and $H$ be as above, $X \subset H$ continuouly and densely.
Let be $u \in L^2(0,T;H)$ such that $u' \in L^2(0,T;X^*)$. For the proof of Theorem \ref{thm:1} we need the following formula of integration by parts
\begin{align}
\label{eq:2.5}
&\scal{\int_0^T \alpha(t) u'(t) \dd t, \varphi}_X \notag\\
&=\,\, \scal{\alpha(T) \tilde{u}(T), \varphi}_X
- \scal{\alpha(0) \tilde{u}(0), \varphi}_X
- \int_0^T \left( \tilde{u}(t), \dot{\alpha}(t) \varphi \right)_H \dd t
\end{align}
for all $\varphi \in X$ and all $\alpha \in C^1([0,T])$ ($\tilde{u}$ denotes the absolutely continuous representative of $u$ with values in $X^*$).
This formula is easily seen by combining the equation
\begin{align*}
\int_0^T \left( \alpha(t) u'(t) + \dot{\alpha}(t) u(t) \right) \dd t\, \stackrel{\text{in } X^*}{=} \alpha(T) \tilde{u}(T) - \alpha(0) \tilde{u}(0)
\end{align*}
and the equation in (\ref{eq:2.4}).

\qed

We consider the following special cases for $X$ and $H$:
\begin{align*}
X = V \quad (\text{resp. } X = V_0), \quad H = L^2(\Omega)^3,
\end{align*}
It follows
\begin{align*}
& L^2(\Omega)^3 \subset V^* \quad (\text{resp. } L^2(\Omega)^3 \subset V_0^*) \quad continuously,\ densely\\[1\baselineskip]
& \scal{z,u}_V\,\, =\,\, (z,u)_{L^2(\Omega)^3} \quad \forall z \in L^2(\Omega)^3, \quad \forall u \in V\\[1\baselineskip]
& (\text{resp. } \scal{z,u}_{V_0}\,\, =\,\, (z,u)_{L^2(\Omega)^3} \quad \forall z \in L^2(\Omega)^3, \quad \forall u \in V_0).
\end{align*}

Without any further reference, in what follows we suppose that $V$ is separable \footnote{The Lipschitz continuity of the boundary $\Gamma$ is sufficient for the density of $C^1(\ol{\Omega})^3$ in $V$.}.
\begin{theorem}
\label{thm:1}
For every weak solution $\{e,h\}$ of (\ref{eq:1.4}), (\ref{eq:1.5}), (\ref{eq:1.10}), (\ref{eq:1.11}) there exist the distributional derivatives
\begin{align}
\label{eq:2.6}
(\ve e)' \in L^2(0,T;V_0^*), \quad (\mu h)' \in L^2(0,T;V^*)
\end{align}
and there holds
\begin{align}
\label{eq:2.7}
& \begin{cases}
&\displaystyle
\scal{(\ve e)'(t),\varphi}_{V_0} + \intl_{\Omega} (-h(x,t) \cdot \curl \varphi(x) + j(x,t) \cdot \varphi(x)) \dd x\,\, =\,\, 0\\
&\text{for a.e. } t \in [0,T] \text{ and all } \varphi \in V_0,
\end{cases}\\[2ex]
\label{eq:2.8}
& \begin{cases}
&\displaystyle
\scal{(\mu h)'(t),\psi}_V + \intl_{\Omega} e(x,t) \cdot \curl \psi(x) \dd x\,\, =\,\, 0\\
&\text{for a.e. } t \in [0,T] \text{ and all } \psi \in V.
\end{cases}
\end{align}
Moreover,
\begin{align}
\label{eq:2.9}
& (\ve e) \in C([0,T];V_0^*), \quad (\mu h) \in C([0,T];V^*),\\[1\baselineskip]
\label{eq:2.10}
& (\ve e)(0)\,\, =\,\, \ve e_0\,\,\, \text{ in } V_0^*, \quad (\mu h)(0)\,\, =\,\, \mu h_0\,\,\, \text{ in } V^*.
\end{align}
\end{theorem}

\begin{remark}
\label{r:2}
\begin{enumerate}[itemindent=1.3em,leftmargin=0em,itemsep=2ex,labelsep=1ex,topsep=1ex,label=\arabic*. ]
\item Let $\mcN$ denote the set of those $t \in [0,T]$ for which the equation in (\ref{eq:2.7}) fails.
\label{r:2.1}
Then $\mes\mcN = 0$, and $\mcN$ does not depend on $\varphi \in V_0$.
This follows from the separability of $V_0$.
An analogous observation is true for (\ref{eq:2.8}).
\item In (\ref{eq:2.9}) and (\ref{eq:2.10}), the absolutely continuous representatives of $\ve e$ and $\mu h$ with respect to the norms in $V_0^*$ and $V^*$, respectively, are understood.
\end{enumerate}
\end{remark}

\begin{proof}[Proof of Theorem \ref{thm:1}]
We identify $\ve e$ with an element in $L^2(0,T;V_0^*)$ and prove the existence of the distributional derivative $(\ve e)' \in L^2(0,T;V_0^*)$.

We define $\mcF \in (L^2(0,T;V_0))^*$ by
\begin{align*}
\scal{\mcF,\phi}_{L^2(V_0)} :=\, \intl_{Q_T} (-h \cdot \curl \phi + j \cdot \phi) \dd x \dd t, \quad \phi \in L^2(0,T;V_0)~\text{\footnotemark}.
\end{align*}
\footnotetext{If there is no danger of confusion, for indices we write $L^p(X)$ in place of $L^p(0,T;X)$.}%
The linear isometry $(L^2(0,T;V_0))^* \cong L^2(0,T;V_0^*)$ permits to identify $\mcF$ with its isometric image in $L^2(0,T;V_0^*)$ which we again denote by $\mcF$.
Thus,
\begin{align*}
\scal{\mcF,\phi}_{L^2(V_0)} =\, \int_0^T \scal{\mcF(t),\phi(t)}_{V_0} \dd t \quad \forall \phi \in L^2(0,T;V_0).
\end{align*}

Given any $\varphi \in V_0$ and $\zeta \in C_c^\infty(]0,T[)$, in (\ref{eq:2.3}) we take $\phi(x,t) = \varphi(x) \zeta(t)$, $\psi(x,t) = 0$ for a.e. $(x,t) \in Q_T$.
It follows
\begin{align*}
\scal{\int_0^T \dot{\zeta}(t) (\ve e)(t),\varphi}_{V_0} \dd t
&\,\, =\, \int_0^T \scal{\dot{\zeta}(t) (\ve e)(t),\varphi}_{V_0} \dd t & \text{(cf.~\cite[p.~421]{Zeidler})}\\[1\baselineskip]
&\,\, =\, \int_0^T \left( \dot{\zeta}(t) (\ve e)(t),\varphi \right)_{L^2(\Omega)^3} \dd t & \text{(by (\ref{eq:2.4}))}\\[1\baselineskip]
&\,\, =\, \intl_{Q_T} (-h \cdot \curl \phi + j \cdot \phi) \dd x \dd t & \text{(by (\ref{eq:2.3}))}\\[1\baselineskip]
&\,\, =\, \scal{\int_0^T \zeta(t) \mcF(t) \dd t,\varphi}_{V_0},
\end{align*}
i.e., $\ve e \in L^2(0,T;L^2(\Omega)^3)$ (when identified with an element in $L^2(0,T;V_0^*)$) possesses the distributional derivative
\begin{align*}
(\ve e)'\,\, =\,\, -\,\mcF \in L^2(0,T;V_0^*),
\end{align*}
and there holds
\begin{align}
\label{eq:2.11}
& \int_0^T \scal{(\ve e)'(t),\varphi}_{V_0} \zeta(t) \dd t \notag\\
& \hspace*{1em} =\, - \int_0^T \scal{(\mcF(t),\varphi \zeta(t)}_{V_0} \dd t \notag\\
& \hspace*{1em} =\, \intl_0^T \left( \intl_{\Omega} (h(x,t) \cdot \curl \varphi(x) - j(x,t) \cdot \varphi(x)) \dd x \right) \zeta(t) \dd t
\end{align}
for any $\varphi \in V_0$ and $\zeta \in C_c^\infty(]0,T[)$.
Thus, by a routine argument, (\ref{eq:2.11}) is equivalent to the equation in (\ref{eq:2.7}) where the set of measure zero of those $t \in [0,T]$ for which this equation is not true,
does not depend on $\varphi$ (cf. Remark \ref{r:2}, \ref{r:2.1}).
By an analogous reasoning one proves $(\mu h)' \in L^2(0,T;V^*)$.

To prove the first equation in (\ref{eq:2.10}), fix $\alpha \in C^1([0,T])$ with $\alpha(0) = 1$, $\alpha(T) = 0$.
We multiply the equation in (\ref{eq:2.7}) by $\alpha(t)$ and integrate over $[0,T]$.
Thus, for any $\varphi \in V_0$,
\begin{align}
\label{eq:2.12}
& \int_0^T \scal{(\ve e)'(t),\varphi \alpha(t)}_{V_0} \dd t \notag\\
& =\, \intl_{Q_T} (h(x,t) \cdot \curl \varphi(x) - j(x,t) \cdot \varphi(x)) \alpha(t) \dd x \dd t \notag\\
& =\, - \intl_{Q_T} (\ve e) \cdot \varphi \dot{\alpha} \dd x \dd t - \intl_{\Omega} (\ve(x) e_0(x)) \cdot \varphi(x) \dd x
\end{align}
On the other hand, the formula of integration by parts (\ref{eq:2.5}) reads
\begin{align}
\label{eq:2.13}
\int_0^T \scal{(\ve e)'(t),\varphi \alpha(t)}_{V_0} \dd t
\,\,=\,\, - \scal{(\widetilde{\ve e})(0),\varphi}_{V_0} - \intl_{Q_T} (\ve e) \cdot \varphi \dot{\alpha} \dd x \dd t.
\end{align}
Combining (\ref{eq:2.12}) and (\ref{eq:2.13}) gives
\begin{align*}
\scal{\ve e_0,\varphi}_{V_0}\,\, =\,\, \scal{(\widetilde{\ve e})(0),\varphi}_{V_0},
\end{align*}
i.e., the first equation in (\ref{eq:2.10}) holds in the sense of $V_0^*$.
An analogous argument gives the second equation in (\ref{eq:2.10}).

\end{proof}

\section{An energy equality}\label{s:3}

For $\{e,h\} \in L^2(Q_T)^3 \times L^2(Q_T)^3$ we define

\begin{align*}
E(t)\,\, =\,\, \frac{1}{2} \intl_{\Omega} \left[ (\ve(x)e(x,t)) \cdot e(x,t) + (\mu(x)h(x,t)) \cdot h(x,t) \right] \dd x \quad \text{for a.e. } t \in [0,T]
\end{align*}
(c.f. Section \ref{s:1.2}).
Then $E \in L^1(0,T).$

The main result of our paper is the following

\begin{theorem}
\label{thm:2}
Assume (\ref{eq:2.1}), (\ref{eq:2.2}) and
\begin{align}
\label{eq:3.1}
\text{$\ve(x)$, $\mu(x)$ are symmetric non-negative matrices for all $x \in \Omega$.}
\end{align}
Then, for every weak solution $\{e,h\}$ of (\ref{eq:1.4}), (\ref{eq:1.5}), (\ref{eq:1.10}), (\ref{eq:1.11}), there exists an absolute continuous function
\begin{align*}
\widetilde{E}: [0,T] \to [0,+\infty[
\end{align*}
such that
\begin{align}
\label{eq:3.2}
&\widetilde{E}(t)\,\, =\,\, E(t) \quad \text{for a.e. } t \in [0,T],\\
\label{eq:3.3}
&\widetilde{E}(t)\,\, =\,\, \widetilde{E}(s) - \intl_s^t \intl_{\Omega} j \cdot e \dd x \dd \tau \quad \forall s,t \in [0,T],\ s < t~\text{\footnotemark},\\
\label{eq:3.4}
&\max_{t \in [0,T]} \widetilde{E}(t)\,\, \leq\,\, \widetilde{E}(0) + \nor{j}_{L^2(Q_T)^3} \nor{e}_{L^2(Q_T)^3},\\
\label{eq:3.5}
&E(t)\,\, =\,\, E(s) - \intl_s^t \intl_{\Omega} j \cdot e \dd x \dd \tau \quad \text{for a.e. } s,t \in [0,T],\ s < t.
\end{align}
\footnotetext{Cf.~(\ref{eq:1.12}).}
\end{theorem}

\noindent\textit{Proof}. We make use of a well-known technique for proving energy inequalities for weak solutions of parabolic equations
by regularizing these solutions in time by the Steklov mean (see, e.g.,~\cite[Chap.~III, §2]{LSU}).
We divide the proof into three parts.
\vspace*{1\baselineskip}

\noindent\textit{$1^\circ$ Integral identities for the Steklov mean of $\{e,h\}$.}\\[1ex]
We extend $\{e,h\}$ by zero for a.e. $(x,t) \in \Omega \times (\R \setminus [0,T])$ and denote this extension by $\{e,h\}$ again.

Let be $\alpha \in C^\infty(\R)$ with $\supp(\alpha) \subset\, ]0,T[$, i.e. there exists $0 < t_0 < t_1 < T$ such that $\alpha(t) = 0$ for all $t \in \, ]-\infty,t_0[ \, \cup \, ]t_1,+\infty[$.
Given $\varphi \in V_0$, for $0 < \lambda < T-t_1$ we consider the function
\begin{align*}
\phi(x,t)\,\, =\,\, \varphi(x) \int_{t-\lambda}^t \alpha(s) \dd s \quad \text{for a.e.} (x,t) \in Q_T.
\end{align*}
Then
\begin{align*}
& \phi(\cdot,t) \in V_0 \quad \text{for all}\,\, t \in [0,T], \quad \phi(x,0)\, =\, \phi(x,T)\, = 0\, \quad \text{for a.e. } x \in \Omega,\\[1\baselineskip]
& \del_t \phi(x,t)\,\, =\,\, \varphi(x) (\alpha(t) - \alpha(t-\lambda)) \quad \text{for a.e. } (x,t) \in Q_T.
\end{align*}\\
Inserting this $\phi$ and the function $\psi = 0$ into (\ref{eq:2.3}) yields

\pagebreak

\begin{align}
\label{eq:3.6}
& \intl_{Q_T} [(\ve e)(x,t+\lambda) - (\ve e)(x,t)] \cdot \varphi(x) \alpha(t) \dd x \dd t \notag\\
& \hspace*{1em} -\, \intl_{Q_T} h(x,t) \cdot \curl \varphi(x) \left( \int_{t-\lambda}^t \alpha(s) \dd s \right) \dd x \dd t \notag\\[2ex]
& \hspace*{1em} +\, \intl_{Q_T} j(x,t) \cdot \varphi(x) \left( \int_{t-\lambda}^t \alpha(s) \dd s \right) \dd x \dd t \notag\\[2ex]
& =\,\, 0
\end{align}
We divide each term of this equation by $\lambda$ ($0 < \lambda < T-t_1$) and make use of Prop. I.1.2, 1.3 (Appendix I).
Then (\ref{eq:3.6}) reads
\begin{align}
\label{eq:3.7}
\intl_{Q_T} \left( \left( \frac{\del}{\del t} (\ve e)_\lambda(x,t) \right) \cdot \varphi(x) - h_\lambda(x,t) \cdot \curl \varphi(x) + j_\lambda(x,t) \cdot \varphi(x) \right) \alpha(t) \dd x \dd t\,\, =\,\, 0,
\end{align}
where
\begin{align*}
f_\lambda(x,t)\, =\, \frac{1}{\lambda} \int_t^{t+\lambda} f(x,s) \dd s, \quad \lambda > 0,~(x,t) \in Q_T
\end{align*}
denotes the \textit{Steklov mean} of $f \in L^p(Q_T)$ ($1 \leq p < +\infty$) (cf.~\cite[Chap.~II, §4]{LSU} and Appendix I below for details).

By an analogous reasoning we conclude from (\ref{eq:2.3}) that
\begin{align}
\label{eq:3.8}
\intl_{Q_T} \left( \left( \frac{\del}{\del t} (\mu h)_\lambda(x,t) \right) \cdot \psi(x) + e_\lambda(x,t) \cdot \psi(x) \right) \alpha(t) \dd x \dd t\,\, =\,\, 0,
\end{align}
for all $\psi \in V$ and all $0 < \lambda < T-t_1$.

To proceed further, we take any sequence $(\lambda_m)_{m \in \N}$ ($0 < \lambda_m < T-t_1$) such that $\lambda_m \to 0$ as $m \to \infty$.
By a routine argument (cf. Section \ref{s:1.3}), from (\ref{eq:3.7}) and (\ref{eq:3.8}) with $\lambda = \lambda_m$ ($m$ fixed) it follows that
\begin{align}
\label{eq:3.9}
&\begin{cases}
&\displaystyle \intl_{\Omega} \left( \left( \frac{\del}{\del t} (\ve e)_{\lambda_m}(x,t) \right) \cdot \varphi(x) - h_{\lambda_m}(x,t) \cdot \curl \varphi(x)
 + j_{\lambda_m}(x,t) \cdot \varphi(x) \right) \dd x\, =\, 0\\[5ex]
&\text{for a.e. } t \in [0,T], \text{ for all } \varphi \in V_0 \text{ and all } m \in \N,
\end{cases}\hspace*{-0.6cm}\\[2ex]
\label{eq:3.10}
&\begin{cases}
&\displaystyle \intl_{\Omega} \left( \left( \frac{\del}{\del t} (\mu h)_{\lambda_m}(x,t) \right) \cdot \psi(x) + e_{\lambda_m}(x,t) \cdot \curl \psi(x) \right) \dd x\,\, =\,\, 0\\[5ex]
&\text{for a.e. } t \in [0,T], \text{ for all } \psi \in V \text{ and all } m \in \N,
\end{cases}\hspace*{-0.6cm}
\end{align}
respectively.
We note that the set of measure zero of those $t \in [0,T]$ for which both (\ref{eq:3.9}) and (\ref{eq:3.10}) fail, is independent of $\varphi \in V_0$, $\psi \in V$ and $m \in \N$.
\vspace*{1\baselineskip}

\noindent\textit{$2^\circ$ $e_\lambda(\cdot,t) \in V_0$, $h_\lambda(\cdot,t) \in V$ for all $0 < \lambda < T-t_1$,
for a.e. $t \in [0,T]$ and all $m \in \N$\footnote{For notational simplicity, in the present part $2^\circ$ of our proof we omit the index $m$ at $\lambda_m$.}.}\\[1ex]
Indeed, (\ref{eq:3.10}) implies
\begin{align*}
\left| \int_{\Omega} e_\lambda(x,t) \cdot \curl \psi(x) \dd x \right|\,\, \leq\,\, \nor{\frac{\del}{\del t}(\mu h)_\lambda(\cdot,t)}_{L^2(\Omega)^3} \nor{\psi}_{L^2(\Omega)^3}
\end{align*}
for all $\psi \in V$.
Thus, by  App. II, Prop. II.1, $e_\lambda(\cdot,t) \in V_0$, i.e., $\curl e_\lambda(\cdot,t) \in L^2(\Omega)^3$ and
\begin{align}
\label{eq:3.11}
\intl_{\Omega} (\curl e_\lambda(x,t)) \cdot \psi(x) \dd x \,\,=\,\, \intl_{\Omega} e_\lambda(x,t) \cdot \curl \psi(x) \dd x \quad \forall \psi \in V
\end{align}
(for a.e. $t \in [0,T]$ and all $m \in N$).
The claim $\curl h_\lambda(\cdot,t) \in L^2(\Omega)^3$ is readily seen by using  Riesz' representation theorem.

We now insert $\psi = h_\lambda(\cdot,t)$ into (\ref{eq:3.11}) and obtain
\begin{align*}
\intl_{\Omega} (\curl e_\lambda(x,t)) \cdot h_\lambda(x,t) \dd x\,\, =\, \intl_{\Omega} e_\lambda(x,t) \cdot \curl h_\lambda(x,t) \dd x.
\end{align*}
On the other hand, since
\begin{align*}
(\ve e)_\lambda(x,t)\,\, =\,\, \ve(x) e_\lambda(x,t), \quad (\mu h)_\lambda(x,t)\,\, =\,\, \mu(x) h_\lambda(x,t) \quad \text{for a.e. } (x,t) \in Q_T,
\end{align*}
we obtain by virtue of the symmetry of $\ve(x)$ and $\mu(x)$
\begin{align}
\label{eq:3.12}
& \frac{1}{2} \frac{\dd}{\dd t} \intl_{\Omega} \left[ (\ve(x) e_\lambda(x,t)) \cdot e_\lambda(x,t) + (\mu(x) h_\lambda(x,t)) \cdot h_\lambda(x,t) \right] \dd x \notag\\
& \hspace*{1em} + \intl_{\Omega} j_\lambda(x,t) \cdot e_\lambda(x,t) \dd x \notag\\
& =\,\, 0
\end{align}
for a.e. $t \in [0,T]$ and all $m \in \N$.

Finally, given any $\zeta \in C_c^\infty(]0,T[)$, we multiply (\ref{eq:3.12}) by $\zeta(t)$ and carry out an integration by parts of the first integral on the left hand side.
Thus,
\begin{align}
\label{eq:3.13}
& -\frac{1}{2} \intl_0^T \left( \intl_{\Omega} \left[ (\ve(x) e_\lambda(x,t)) \cdot e_\lambda(x,t) + (\mu(x) h_\lambda(x,t)) \cdot h_\lambda(x,t) \right] \dd x \right) \dot{\zeta}(t) \dd t \notag\\[0.7\baselineskip]
& \hspace*{1em} + \intl_0^T \left( \intl_{\Omega} j_\lambda(x,t) \cdot e_\lambda(x,t) \dd x \right) \zeta(t) \dd t \notag\\[0.7\baselineskip]
& = \,\,0
\end{align}
for all $m \in \N$.

\pagebreak

\noindent\textit{$3^\circ$ Passing to limits $m \to \infty$.}\\[1ex]
Observing that
\begin{align*}
e_{\lambda_m} \to e, \quad h_{\lambda_m} \to h, \quad j_{\lambda_m} \to j \quad \text{in } L^2(Q_T)^3 \quad\text{ as } \,m \to \infty
\end{align*}
(cf. App. I, Prop. I.2), the passage to limits $m \to \infty$ in (\ref{eq:3.13}) (with $\lambda = \lambda_m$) gives
\begin{align*}
- \int_0^T E(t) \dot{\zeta}(t) \dd t + \int_0^T \left( \intl_{\Omega} j(x,t) \cdot e(x,t) \dd x \right) \zeta(t) \dd t\,\, =\,\, 0 \quad \forall \zeta \in C_c^\infty(]0,T[).
\end{align*}
It follows that the equivalence class $E \in L^1(0,T)$ possesses an absolutely continuous representative $\widetilde{E}: [0,T] \to [0,+\infty[$ such that
\begin{align*}
\widetilde{E}(t)\,\, =\,\, \widetilde{E}(s) - \intl_s^t \intl_{\Omega} j(x,t) \cdot e(x,t) \dd x \dd t \quad \forall s, t \in [0,T], \ s < t,
\end{align*}
i.e., (3.3) holds.

Estimate (\ref{eq:3.4}) and the equality in (\ref{eq:3.5}) are direct consequences of (\ref{eq:3.2}) and (\ref{eq:3.3}).

The proof of Theorem~\ref{thm:2} is complete.

\section{Uniqueness of weak solutions}\label{s:4}

Let be $\ve(x)$, $\mu(x)$ ($x \in \Omega$) as in (\ref{eq:2.1}) and (\ref{eq:3.1}).
In addition, suppose that
\begin{align}
\label{eq:4.1}
\begin{cases}
&\exists\, \ve_* = \text{const} > 0,~\mu_* = \text{const} > 0,~such~that\\[2ex]
&(\ve(x) \xi) \cdot \xi\, \geq\, \ve_* |\xi|^2,~(\mu(x) \xi) \cdot \xi\, \geq\, \mu_* |\xi|^2 \quad \forall x \in \Omega,~\forall \xi \in \R^3.
\end{cases}
\end{align}

We consider equ. (\ref{eq:1.4}) with Ohm law
\begin{align*}
j\,\, =\,\, \sigma e,
\end{align*}
where the entries of the matrix $\sigma = \{ \sigma_{kl}(x) \}_{k,l=1,2,3}$ ($x \in \Omega$) are bounded measurable functions in $\Omega$.
\begin{theorem}
\label{thm:3}
Suppose that the matrices $\ve(x)$, $\mu(x)$ ($x \in \Omega$) satsify (\ref{eq:2.1}), (\ref{eq:3.1}) and (\ref{eq:4.1}).

Let $\{ e,h \} \in L^2(Q_T)^3 \times L^2(Q_T)^3$ be a weak solution of (\ref{eq:1.4}), (\ref{eq:1.5}), (\ref{eq:1.10}), (\ref{eq:1.11}) with initial data
\begin{align*}
e_0\,\, =\,\, h_0\,\, = 0 \quad a.e.~in~\Omega.
\end{align*}
Then
\begin{align*}
e\,\, =\,\, h =\,\, 0 \quad a.e.~in~Q_T.
\end{align*}
\end{theorem}

To prove this theorem, we first derive an energy inequality for the primitives of the functions $t \mapsto e(x,t)$, $t \mapsto \mu(x,t)$ ($x \in \Omega$)
(cf. the proof of Theorem \ref{thm:2} above; see also \cite[Sect 7.8.2]{Fabrizio}).
From this inequality the claim $e = h = 0$ a.e. in $Q_T$ follows easily
(cf. also \cite[Chap.~VII,~4.3]{Duvaut} for a different argument).
A uniqueness result for solutions of linear second order evolution equations in Hilbert spaces has been proved in \cite[Chap.~3,~8.2]{Lions}.
\begin{proof}[Proof of Theorem \ref{thm:3}]
Given any $t \in ]0,T[$, we consider (\ref{eq:2.7}) for a.e. $s \in [0,t]$, multiply this identity by $t-s$, integrate over the interval $[0,t]$ and use an integration by parts in the integral
\begin{align*}
\int_0^t \scal{(\ve e)'(s),(t-s)\varphi}_{V_0} \dd s
\end{align*}\\
(observe $(\ve e)(0) = \ve e_0 = 0$ in $V_0^*$; cf. (\ref{eq:2.10})).
It follows
\begin{align}
\label{eq:4.2}
&\int_0^t \intl_\Omega ((\ve e)(x,s)) \cdot \varphi(x) \dd x \dd s \notag\\[1\baselineskip]
&+ \int_0^t \intl_\Omega (-h(x,s) \cdot \curl \varphi(x) + (\sigma e)(x,s) \cdot \varphi(x))(t-s) \dd x \dd s \notag\\[1\baselineskip]
&=\,\, 0
\end{align}\\
for all $\varphi \in V_0$.
Differentiating each term on the left hand side with respect to $t$ we find 
\begin{align}
\label{eq:4.3}
&\intl_\Omega ((\ve e)(x,t)) \cdot \varphi(x) \dd x \notag\\[1\baselineskip]
&+ \int_0^t \intl_\Omega (-h(x,s) \cdot \curl \varphi(x) + (\sigma e)(x,s) \cdot \varphi(x)) \dd x \notag\\[1\baselineskip]
&=\,\, 0
\end{align}\\
for all $\varphi \in V_0$ and for a.e. $t \in [0,T]$.
From (\ref{eq:2.8}) we obtain analogously
\begin{align}
\label{eq:4.4}
\intl_\Omega (\mu h)(x,t) \cdot \psi(x) \dd x + \int_0^t \intl_\Omega e(x,s) \cdot \curl \psi(x) \dd x \dd s\,\, =\,\, 0
\end{align}
for all $\psi \in V$ and for a.e. $t \in [0,T]$.

For $t \in [0,T]$ and for a.e. $x \in \Omega$ we define
\begin{align*}
\hat{e}(x,t) :=\, \int_0^t e(x,s) \dd s, \quad \hat{h}(x,t) :=\, \int_0^t h(x,s) \dd s.
\end{align*}
Then $\hat{e}, \hat{h} \in L^2(Q_T)^3$, and the weak time-derivatives of these functions are $\del_t \hat{e} = e$, $\del_t \hat{h} = h$ a.e. in $Q_T$
(cf. App. I, Prop. I.1.2).
Using Fubini's theorem, (\ref{eq:4.3}) and (\ref{eq:4.4}) can be rewritten in the form
\begin{align}
\label{eq:4.5}
&\intl_\Omega \hat{h}(x,t) \cdot \curl \varphi(x) \dd x
\,\,=\, \intl_\Omega ((\ve e)(x,t) + (\sigma \hat{e})(x,t)) \cdot \varphi(x) \dd x \quad \forall \varphi \in V_0,\\[1\baselineskip]
\label{eq:4.6}
&\intl_\Omega \hat{e}(x,t) \cdot \curl \psi(x) \dd x
\,\,=\, -\intl_\Omega (\mu h )(x,t) \cdot \psi(x) \dd x \quad \forall \psi \in V
\end{align}
for a.e. $t \in [0,T]$, respectively.

From (\ref{eq:4.5}) and (\ref{eq:4.6}) we conclude 
\begin{align*}
&\curl \hat{h}(\cdot,t) \in V \quad \text{[by Riesz' representation theorem],}\\[1\baselineskip]
&\curl \hat{e}(\cdot,t) \in V_0 \quad \text{[by App. II, Prop. II.1]}
\end{align*}
(i.e., $\curl \hat{e}(\cdot,t) \in V$, and 
\begin{align}
\label{eq:4.7}
\intl_\Omega (\curl \hat{e}(x,t)) \cdot \psi(x) \dd x\,\, =\, \intl_\Omega \hat{e}(x,t) \cdot \curl \psi(x) \dd x \quad \forall \psi \in V).
\end{align}
respectively.
Thus, $\varphi = \hat{e}(\cdot, t)$ and $\psi = \hat{h}(\cdot,t)$ are admissible test functions in (\ref{eq:4.5}) and (\ref{eq:4.6}).
Adding then these equations and observing (\ref{eq:4.7}) we find
\begin{align}
\label{eq:4.8}
\intl_\Omega ((\ve e)(x,t)) \cdot \hat{e}(x,t) \dd x
+ \intl_\Omega (\sigma \hat{e})(x,t) \cdot \hat{e}(x,t) \dd x
+ \intl_\Omega (\mu h)(x,t) \cdot \hat{h}(x,t) \dd x
\,\,=\,\, 0
\end{align}
for a.e. $t \in [0,T]$.

To proceed, we notice that for every $x \in \Omega$ the functions $t \mapsto \hat{e}(x,t)$, $t \mapsto \hat{h}(x,t)$
are H\"older continuous with exponent $\frac{1}{2}$ on the interval $[0,T]$.
By (\ref{eq:3.1}) (symmetry of $\ve(x)$, $\mu(x)$),
\begin{align*}
&\frac{1}{2} \frac{\dd}{\dd t} \intl_\Omega ((\ve \hat{e})(x,t) \cdot \hat{e}(x,t) + (\mu \hat{h})(x,t) \cdot \hat{h}(x,t)) \dd x\\[1\baselineskip]
&= \intl_\Omega ((\ve e)(x,t) \cdot \hat{e}(x,t) + (\mu h)(x,t) \cdot \hat{h}(x,t)) \dd x
\end{align*}
for a.e. $t \in [0,T]$.
From (\ref{eq:4.8}) it follows by integration that
\begin{align}
\label{eq:4.9}
&\frac{1}{2} \intl_\Omega ((\ve \hat{e})(x,t) \cdot \hat{e}(x,t) + (\mu \hat{h})(x,t) \cdot \hat{h}(x,t)) \dd x \notag\\[1\baselineskip]
&=\, -\int_0^t \intl_\Omega ((\sigma \hat{e})(x,s) \cdot \hat{e}(x,s) \dd x \dd s
\end{align}
for all $t \in [0,T]$.
Observing (\ref{eq:4.1}) we derive from (\ref{eq:4.9}) by the aid of Gronwall's inequality
\begin{align*}
\intl_\Omega (|\hat{e}(x,t)|^2 + |\hat{h}(x,t)|^2) \dd x\,\, =\,\, 0 \quad \forall t \in [0,T].
\end{align*}
Thus, by Fubini's theorem,
\begin{align*}
\int_0^t e(x,s) \dd s\,\, = \int_0^t h(x,s) \dd s\,\, =\,\, 0 \quad \text{for a.e. } (x,t) \in Q_T.
\end{align*}
We extend $e$, $h$ by zero onto $\Omega \times [T,+\infty[$ and denote this extension by these letters again. 
Hence, for all $\lambda > 0$,
\begin{align*}
\int_t^{t+\lambda} e(x,s) \dd s\,\, =\, \int_t^{t+\lambda} h(x,s) \dd s\,\, =\,\, 0 \quad \text{for a.e. } (x,t) \in Q_T,
\end{align*}
i.e., the Steklov means $e_\lambda$, $h_\lambda$ vanish a.e. in $Q_T$.
Whence
\begin{align*}
e\,\, =\,\, h\,\, = 0 \quad \text{a.e. in } Q_T
\end{align*}
(see App. I, Prop. I.2).

\end{proof}

\appendix
\renewcommand{\thesection}{\Roman{section}}

\section{The Steklov mean of an $L^p$ function}\label{s:AppI}

Let $\Omega \subseteq \R^N$ ($N \geq 2$) be any open set, let $0 < T < +\infty$ and put $Q_T = \Omega \times ]0,T[$.

Let $f \in L^p(Q_T)$ ($1 \leq p < +\infty$).
We extend $f$ by zero a.e. onto $\Omega \times ]T,+\infty[$ and denote this function by $f$ again.
For $\lambda > 0$, the function
\begin{align*}
f_\lambda(x,t)\, =\, \frac{1}{\lambda} \int_t^{t+\lambda} f(x,s) \dd s \quad \text{for a.e. } (x,t) \in Q_T.
\end{align*}
is called \textit{Steklov mean of $f$}.
\begin{propo}
\label{prop:I}
For every $f \in L^p(Q_T)$ and every $\lambda > 0$ there holds
\begin{enumerate}[itemindent=4em,leftmargin=0em,itemsep=2ex,label=\ref{prop:I}.\arabic*,labelsep=2em,topsep=1ex]
\item
\label{prop:I.1}
$\displaystyle \nor{f_\lambda}_{L^p(Q_T)}\, \leq\, \nor{f}_{L^p(Q_T)}$;
\item
\label{prop:I.2}
$\displaystyle \intl_{Q_T} f_\lambda(x,t) \del_t \zeta(x,t) \dd x \dd t\, =\, -\,\frac{1}{\lambda} \intl_{Q_T} (f(x,t+\lambda) - f(x,t)) \zeta(x,t) \dd x \dd t$
\\[1\baselineskip]
for every $\zeta \in C_c^\infty(Q_T)$, i.e., $f_\lambda$ possesses the weak derivative
\begin{align*}
\frac{\del}{\del t} f_\lambda(x,t)\, =\, \frac{1}{\lambda} (f(x,t+\lambda) - f(x,t))
\end{align*}
for a.e. $(x,t) \in Q_T$.
\item
\label{prop:I.3}
$\displaystyle \frac{1}{\lambda} \intl_{Q_T} f(x,t) \left( \int_{t-\lambda}^t \alpha(s) \dd s \right) \dd x \dd t\, =\, \intl_{Q_T} f_\lambda(x,t) \alpha(t) \dd x \dd t$
\\[1\baselineskip]
for any $\alpha \in L^\infty(\R)$ such that $\alpha(t) = 0$ for all $t \in \R \setminus [t_0,t_1]$ \textnormal{(}$0 < t_0 < t_1 < T$ depending on $\alpha$\textnormal{)}.
\end{enumerate}
\end{propo}
\begin{propo}
\label{prop:II}
For every $f \in L^p(Q_T)$,
\begin{align*}
\lim_{\lambda \to 0} \nor{f_\lambda - f}_{L^p(Q_T)}\, =\, 0.
\end{align*}
\end{propo}
Prop.~\ref{prop:I.1} and Prop.~\ref{prop:II} are special cases of well-known results about the mollification of $L^p$-functions.
The properties of the Steklov mean presented in Prop.~\ref{prop:I.2}, \ref{prop:I.3} are used in \cite[Chap.~III, §2, Lemma~2.1]{LSU} to establish energy inequalities for weak solutions of parabolic equations.
For reader's convenience we present the proofs.
\begin{proof}[Proof of Proposition \ref{prop:I.2}]
Let $\zeta \in C_c^\infty(Q_T)$.
Then
\begin{align*}
- \intl_{Q_T} f_\lambda(x,t) \frac{\del \zeta}{\del t}(x,t) \dd x \dd t \,=\, \lim_{h \to 0} \intl_{Q_T} f_\lambda(x,t) \frac{1}{h} (\zeta(x,t-h) - \zeta(x,t)) \dd x \dd t.
\end{align*}
On the other hand, for $0 < h < \lambda$ we have
\begin{align*}
& \intl_{Q_T} f_\lambda(x,t) \frac{1}{h} (\zeta(x,t-h) - \zeta(x,t)) \dd x \dd t\\[2ex]
& =\, \frac{1}{\lambda h} \intl_{Q_T} \left( \int_{t+\lambda}^{t+\lambda+h} f(x,s) \dd s - \int_t^{t+h} f(x,s) \dd s \right) \zeta(x,t) \dd x \dd t\\[2ex]
& =\, \frac{1}{\lambda h} \intl_{Q_T} \left( \int_{t+\lambda}^{t+\lambda+h} (f(x,s) - f(x,t+\lambda)) \dd s \right) \zeta(x,t) \dd x \dd t\\[1\baselineskip]
& \hspace*{1em}\, - \frac{1}{\lambda h} \intl_{Q_T} \left( \int_t^{t+h} (f(x,s) - f(x,t)) \dd s \right) \zeta(x,t) \dd x \dd t\\[1\baselineskip]
& \hspace*{1em}\, + \frac{1}{\lambda} \intl_{Q_T} (f(x,t+\lambda) - f(x,t)) \zeta(x,t) \dd x \dd t.
\end{align*}
Here, the first and the second term on the right hand side converge to zero when $h \to 0$.
This can be easily seen by combinig Fubini's theorem and continuity of $f$ with respect to the integral mean.

Whence the claim.

\end{proof}
\noindent\textit{Proof of Proposition \ref{prop:I.3}.~}
Let $\lambda > 0$.
We introduce a function $\xi_\lambda: \R^2 \to \{0\} \cup \{1\}$ as follows:
given $t \in \R$, define
\begin{align*}
\xi_\lambda(s,t) =
\begin{cases}
1 & \text{if } s \in [t-\lambda,t],\\
0 & \text{if } s \in \R \setminus [t-\lambda,t],
\end{cases}
\end{align*}
or, equivalently, given $s \in \R$, define
\begin{align*}
\xi_\lambda(s,t) =
\begin{cases}
1 & \text{if } t \in [s,s+\lambda],\\
0 & \text{if } t \in \R \setminus [s,s+\lambda].
\end{cases}
\end{align*}
We obtain for a.e. $x \in \Omega$
\begin{align*}
& \int_0^T f(x,t) \left( \int_{t-\lambda}^t \alpha(s) \dd s \right) \dd t\\[2ex]
& =\, \int_0^T \intl_{\R} f(x,t) \xi_\lambda(s,t) \alpha(s) \dd s \dd t\\[2ex]
& =\, \intl_{\R} \int_0^T f(x,t) \xi_\lambda(s,t) \alpha(s) \dd t \dd s && \text{(by Fubini's theorem)}\\[2ex]
& =\, \int_0^T \left( \int_s^{s+\lambda} f(x,t) \dd t \right) \alpha(s) \dd s && \text{(since $\alpha = 0$ on $\R \setminus [t_0,t_1]$).}
\end{align*}
Integrating over $\Omega$ and dividing by $\lambda$ gives the claim.

\section{Equivalent characterization of the space $V_0$}\label{s:AppII}

Let $\Omega \subset \R^3$ be an open set.
We recall the definition of the spaces $V$ and $V_0$ introduced in Section~\ref{s:2}
\begin{align*}
V &:= \left\{ u \in L^2(\Omega^3); \ \curl u \in L^2(\Omega)^3 \right\},\\
V_0 &:= \left\{ u \in V; \ \intl_{\Omega} (\curl u) \cdot z \dd x\, = \intl_{\Omega} u \cdot (\curl z) \dd x \,\,\, \forall z \in V \right\}.
\end{align*}
\begin{propo}
\label{prop:II.1}
Let $\Omega \subset \R^3$ be a bounded domain with Lipschitz boundary $\Gamma = \del \Omega$.
Then the following are equivalent
\begin{enumerate}[itemindent=3em,leftmargin=0em,itemsep=2ex,label=$\arabic*^\circ$,labelsep=2em,topsep=1ex]
\item $u \in V_0$;
\item $u \in L^2(\Omega)^3$, $\exists\, c = \textnormal{const}\, >\, 0$ such that
\,\,$\displaystyle \left| \intl_{\Omega} u \cdot \curl z \dd x \right|\, \leq\, c \nor{z}_{(L^2)^3}\,\,\,\forall z \in V$.
\end{enumerate}
\end{propo}
\noindent The implication $1^\circ \Rightarrow 2^\circ$ is easily seen.
To prove the reverse implication, we will use the following
\begin{lemma}
\label{lemma:1}
Let $\Omega \subset \R^3$ be a bounded domain with Lipschitz boundary $\Gamma$.
Let $w^* \in H^{-1/2}(\Gamma)$ satisfy
\begin{align}
\label{eq:II.1}
\exists\, c_0 = \textnormal{const}\, > 0\, \text{ such that }\,\, \left| \scal{w^*, u}_{H^{1/2}} \right|\, \leq \,c_0 \nor{u}_{L^2} \,\,\, \forall u \in H^1(\Omega).
\end{align}
Then
\begin{align*}
w^* = 0.
\end{align*}
\end{lemma}
Indeed, if $w^* \not= 0$, then there would exist $u_0 \in H^1(\Omega)$ such that
\begin{align*}
\scal{w^*, u_0}_{H^{1/2}} \, \not=\, \,0.
\end{align*}
We then take an open set $\Omega' \subset \Omega$ and a function $\zeta \in C^1(\ol{\Omega})$ with the following properties
\begin{itemize}[topsep=1\baselineskip]
\setlength{\itemsep}{2ex}
\item\, $\displaystyle \ol{\Omega'} \subset {\Omega}, \quad \intl_{\Omega \setminus \Omega'} u_0^2 \dd x
\,\leq \,\left( \dfrac{1}{2c_0} \left| \scal{w^*, u_0}_{H^{1/2}} \right| \right)^2$,
\item\, $0\, \leq \zeta \leq\, 1\,\,\,$ in $\Omega$,\,\,\, $\zeta\, =\, 0\,\,\,$ in $\Omega'$,\,\,\, $\zeta\, =\, 1\,\,\,$ on $\Gamma$.
\end{itemize}
It follows
\begin{align*}
\left| \scal{w^*, u_0}_{H^{1/2}} \right|
&\, =\,\left| \scal{w^*, \zeta u_0}_{H^{1/2}} \right|\\[0.7\baselineskip]
&\, \leq\, c_0 \left( \ \intl_{\Omega \setminus \Omega'} u_0^2 \dd x \right)^{1/2} & \text{(by (\ref{eq:II.1}))}\\[0.7\baselineskip]
&\, \leq\, \dfrac{1}{2} \left| \scal{w^*, u_0}_{H^{1/2}} \right|,
\end{align*}
a contradiction.
\begin{proof}[Proof of $2^\circ \Rightarrow 1^\circ$]
From $2^\circ$ one concludes by the aid of Riesz' representation theorem that $\curl u \in L^2(\Omega)^3$.
We obtain
\begin{align*}
\intl_\Omega (\curl u) \cdot z \dd x - \intl_\Omega u \cdot (\curl z) \dd x\,\, =\,\, \scal{\gamma_\tau(u), z}_{(H^{1/2})^3} \quad \forall z \in H^1(\Omega)^3
\end{align*}
(cf., e.g., \cite[p. 207]{Dautray}; recall $\gamma_\tau(u) = n \times u |_\Gamma$ if $u \in C^1(\ol{\Omega})^3$, cf. Section~\ref{s:2}).
Thus, for all $z \in H^1(\Omega)^3$,
\begin{align*}
\left| \scal{\gamma_\tau(u), z}_{(H^{1/2})^3} \right|\, \leq\, \left( \nor{\curl u}_{(L^2)^3} + c \right) \nor{z}_{(L^2)^3} \quad \text{(by $2^\circ$)}.
\end{align*}
By the above lemma,
\begin{align*}
\gamma_\tau(u)\,\, =\,\, 0 \text{ in } H^{-1/2}(\Gamma)^3.
\end{align*}
The density of $C^1(\ol{\Omega})^3$ in $V$ (cf. \cite[p. 204]{Dautray}, \cite[Chap. VII, Lemme 4.1]{Duvaut}) implies
\begin{align*}
\intl_\Omega (\curl u) \cdot z \dd x - \intl_\Omega u \cdot (\curl z) \dd x\,\, =\,\, 0 \quad \forall z \in V,
\end{align*}
i.e., $u \in V_0$.

\end{proof}
\noindent From Prop.~\ref{prop:II.1} we conclude\\[1\baselineskip]
\textit{%
For $\{u, v\} \in L^2(\Omega)^3 \times L^2(\Omega)^3$ the following are equivalent
\begin{enumerate}[itemindent=0.0em,leftmargin=2em,itemsep=2ex,labelsep=1ex,topsep=1ex,label=(\alph*) ]
\item $\{u, v\} \in V \times V_0$;
\item $\exists\, c\, =\, \textnormal{const} > 0$ such that\\[1\baselineskip]
$\displaystyle \left| \intl_\Omega \left( -u \cdot \curl \varphi + v \cdot \curl \psi \right) \dd x \right|
\,\leq\, c \left( \nor{\varphi}_{(L^2)^3}^2 + \nor{\psi}_{(L^2)^3}^2 \right)^{1/2} \quad \forall \{\varphi, \psi\} \in V_0 \times V$.
\end{enumerate}
}


\allowdisplaybreaks[4]

\end{document}